\documentclass[amssymb, 11pt]{amsart}
\usepackage{latexsym, amssymb}

\newcommand{\cN}{\mathcal{N}}
\newcommand{\cM}{\mathcal{M}}
\newcommand{\R}{\mathbb{R}}
\newcommand{\Z}{\mathbb{Z}}
\newcommand{\F}{\mathbb{F}}
\newcommand{\Irr}{\mathrm{Irr}}
\newcommand{\dgr}{\mathrm{deg}}
\newcommand{\vep}{\varepsilon}
\newcommand{\SO}{\mathrm{SO}}
\newcommand{\Sp}{\mathrm{Sp}}
\newcommand{\GL}{\mathrm{GL}}
\newcommand{\GU}{\mathrm{U}}
\newcommand{\cP}{\mathcal{P}}
\newcommand{\cS}{\mathcal{S}}
\newcommand{\bG}{\mathbf{G}}
\newcommand{\bT}{\mathbf{T}}

\newtheorem{theorem}{Theorem}[section]
\newtheorem{proposition}{Proposition}[section]
\newtheorem{lemma}{Lemma}[section]

\numberwithin{equation}{section}

\allowdisplaybreaks

\begin{document}

\title [Real representations of finite symplectic groups]{Real representations of finite symplectic groups over fields of characteristic two}

\author{C. Ryan Vinroot}
\address{Department of Mathematics\\College of William and Mary\\ Williamsburg, VA, 23187\\USA}
\email{vinroot@math.wm.edu}


\begin{abstract}  We prove that when $q$ is a power of $2$, every complex irreducible representation of $\Sp(2n, \F_q)$ may be defined over the real numbers, that is, all Frobenius-Schur indicators are 1.  We also obtain a generating function for the sum of the degrees of the unipotent characters of $\Sp(2n, \F_q)$, or of $\SO(2n+1,\F_q)$, for any prime power $q$.
\end{abstract}

\maketitle

\thispagestyle{empty}

\section{Introduction}

It was proved by Gow \cite[Theorem 1]{Gow85} that if $q$ is the power of an odd prime, and $\chi$ is an irreducible complex character of the finite symplectic group $\Sp(2n,\F_q)$ which is real-valued, then the complex representation $\pi$ which affords $\chi$ is a real representation (or defined over $\R$) if and only if $\pi$ acts trivially on the center of $\Sp(2n, \F_q)$.  That is, Gow computes the Frobenius-Schur indicators of the real-valued characters of $\Sp(2n,\F_q)$ when $q$ is odd.  As it turns out, when $q \equiv 1($mod $4)$ all irreducible complex characters of $\Sp(2n, \F_q)$ are real-valued, but this is not the case when $q \equiv 3($mod $4)$.  This author unified these cases by proving \cite[Theorem 1.3]{V05} that for any odd $q$, there is a certain twisted Frobenius-Schur indicator which is $1$ for all irreducible complex characters of $\Sp(2n, \F_q)$. 

The technique used to prove the results just stated is not amenable for the case when $q$ is even, although several partial results have been proved.  It follows from results of Gow \cite{Gow81} and Ellers and Nolte \cite{ElNo82} that when $q$ is even, all complex irreducible characters of $\Sp(2n, \F_q)$ are real-valued.  A result of Prasad \cite[Theorem 3]{Pr98} implies that any complex irreducible representation of $\Sp(2n, \F_q)$ with $q$ even, which appears in the Gelfand-Graev representation, can be defined over the real numbers.  It has also been shown computationally \cite[Theorem 5.2]{TV17} than when $q$ is even and $n \leq 8$, all complex irreducible representations of $\Sp(2n, \F_q)$ are real representations.  In this paper, we finally settle the general case, and we show in Theorem \ref{main2} that when $q$ is even every complex irreducible representation of $\Sp(2n, \F_q)$, for any $n$, is a real representation.  In the process of proving this statement, we also obtain a new combinatorial identity in Theorem \ref{main1}, which states that for any prime power $q$, the sum of the degrees of the unipotent characters of $\Sp(2n,\F_q)$ (or of $\SO(2n+1, \F_q)$) is $\prod_{i=1}^n (q^{2i}-1)$ times the coefficient of $u^n$ in the expansion of
$$ \prod_{i \geq 1} \frac{1 + u/q^{2i}}{1 - u/q^{2i-1}} \prod_{1 \leq i < j \atop{ i + j \text{ odd}}} (1 - u^2/q^{i+j})^{-1}.$$

We now outline our method of proof, as well as the structure of this paper.  Preliminary results and notation are given in Section \ref{Prelim}.  By the Frobenius-Schur involution formula, given in Section \ref{FS}, a finite group $G$ has the property that all of its complex irreducible representations are defined over the real numbers (that is, all Frobenius-Schur indicators are 1) if and only if the sum of the character degrees of $G$ is equal to the number of involutions in $G$.  The strategy is to develop a generating function for the character degree sum of $\Sp(2n, \F_q)$, where $q$ is even, and match this with the generating function for the number of involutions in $\Sp(2n, \F_q)$, which was computed by Fulman, Guralnick, and Stanton \cite{FGS}.  The main obstruction to directly calculating the character degree sum is the lack of combinatorial control of the sum of degrees of unipotent characters.  To work around this obstruction, we apply another result of Gow \cite[Theorem 2]{Gow85} that when $q$ is odd, all complex irreducible characters of $\SO(2n+1, \F_q)$ have Frobenius-Schur indicator $1$.  We then take advantage of the following three facts:
\begin{enumerate}
\item[(i)] When $q$ is odd, the generating function for the character degree sum of $\SO(2n+1, \F_q)$ is given by the generating function for the number of involutions in this group (by Gow's result).
\item[(ii)] The unipotent character degrees of $\Sp(2n, \F_q)$, or of $\SO(2n+1, \F_q)$, are given by expressions in $q$ which are independent of whether $q$ is odd or even.
\item[(iii)] When $q$ is even, $\SO(2n+1, \F_q) \cong \Sp(2n, \F_q)$.
\end{enumerate}
In Section \ref{chardegrees}, we describe all character degrees of $\SO(2n+1, \F_q)$, for any $q$, and in Proposition \ref{GenFunSO} we give a generating function for the character degree sum of this group.  In Section \ref{Main} we obtain our main results by first expanding, in Proposition \ref{GUexpand}, the factors in the generating function from Proposition \ref{GenFunSO} corresponding to unipotent degrees of finite general linear and unitary groups, by adapting methods used in \cite{FV14}.  The remaining factor corresponds to the sum of unipotent character degrees of $\Sp(2n, \F_q)$.  Using fact (i) above, we solve for this factor of the generating function when $q$ is odd, giving us Theorem \ref{main1}.  By fact (ii), this factor is also the generating function for the unipotent character degree sum of $\Sp(2n, \F_q)$ when $q$ is even, and we can use this and fact (iii) to finally match the character degree sum of $\SO(2n+1, \F_q) \cong \Sp(2n, \F_q)$ with the number of involutions when $q$ is even, yielding Theorem \ref{main2}.

Also relevant to our main result in Theorem \ref{main2} is the following conjecture regarding finite simple groups, first stated explicitly in \cite{KaKu}, and discussed in \cite[Section 3]{TV17}:  A finite simple group $G$ has the property that all of its complex irreducible representations are real representations, if and only if every element of $G$ is the product of two involutions from $G$.  Our Theorem \ref{main2} resolves one of the remaining families to check for this conjecture to hold.  We hope to extend the methods below to confirm this conjecture for the final cases in a subsequent paper.\\
\\
\noindent {\bf Acknowledgements. }  The author was supported in part by a grant from the Simons Foundation, Award \#280496.

\section{Preliminaries} \label{Prelim}

\subsection{Frobenius-Schur indicators and involutions} \label{FS}

If $G$ is a finite group, we let $\Irr(G)$ denote the set of irreducible complex characters of $G$.  Given $\chi \in \Irr(G)$, the {\em Frobenius-Schur indicator} of $\chi$, which we denote by $\vep(\chi)$, is defined by the formula
$$ \vep(\chi) = \frac{1}{|G|} \sum_{g \in G} \chi(g^2).$$
Let $(\pi, V)$ be an irreducible complex representation afforded by the character $\chi$.  We say that $(\pi, V)$ is a {\em real representation} or {\em defined over $\R$} if there is a basis for $V$ such that all images of the corresponding matrix representation have all real entries.  By classical results of Frobenius and Schur, $\vep(\chi)$ has the following property:
$$ \vep(\chi) = \left\{ \begin{array}{rl} 1 & \text{if $(\pi, V)$ is defined over $\R$}, \\
-1 & \text{if $\chi = \bar{\chi}$ but $(\pi, V)$ is not defined over $\R$,} \\
0 & \text{if } \chi \neq \bar{\chi}.
\end{array} \right.$$
By applying this fact and the orthogonality relations of characters, one obtains the {\em Frobenius-Schur involution formula}:
$$ \sum_{\chi \in \Irr(G)} \vep(\chi) \chi(1) = \# \{ g \in G \, \mid \, g^2 = 1\}.$$

Define an {\em involution} in $G$ to be an element $g \in G$ such that $g^2 = 1$.  It then follows from the formula above that $\vep(\chi) = 1$ for every $\chi \in \Irr(G)$ if and only if the sum of all character degrees of $G$ is equal to the number of involutions in $G$.  See \cite[Chapter 4]{Isaacs} for all of the above theory regarding the Frobenius-Schur indicator.

We now consider the finite groups of interest.  Throughout this paper, we let $\F_q$ be a finite field with $q$ elements with $q$ a power of a prime $p$, and $\bar{\F}_q$ a fixed algebraic closure of $\F_q$.  First let $G = \mathrm{SO}(2n+1, \F_q)$ be the finite special orthogonal group over $\F_q$, where $q$ is odd.  Gow \cite[Theorem 2]{Gow85} proved that $\vep(\chi) = 1$ for every $\chi \in \Irr(G)$.  Meanwhile, Fulman, Guralnick, and Stanton have computed a generating function for the number of involutions in $G =\mathrm{SO}(2n+1, \F_q)$ (where $G=1$ if $n=0$).  Specifically, it follows from \cite[Theorem 2.17 and Lemma 6.1(3)]{FGS} and the fact that $\mathrm{O}(2n+1, \F_q) \cong \SO(2n+1, \F_q) \times \{ \pm 1 \}$ (see also \cite[Theorem 7.1(2)]{TV17}), that the number of involutions in $\SO(2n+1, \F_q)$ is $\prod_{i=1}^n (q^{2i} - 1)$ times the coefficient of $u^n$ in the expansion of
$$\frac{1}{1-u} \prod_{i \geq 1} \frac{(1 + u/q^{2i})^2}{1-u^2/q^{2i}}.$$
From this result, along with the Frobenius-Schur involution formula and Gow's result that $\vep(\chi) = 1$ for every $\chi \in \Irr(\SO(2n+1, \F_q))$, we have the following consequence.

\begin{lemma} \label{SOqodd}
For $q$ odd, the sum of the character degrees of $\SO(2n+1, \F_q)$ is $\prod_{i=1}^{n} (q^{2i}-1)$ times the coefficient of $u^n$ in the expansion of
$$\frac{1}{1-u} \prod_{i \geq 1} \frac{(1 + u/q^{2i})^2}{1-u^2/q^{2i}}.$$
\end{lemma}

We now consider the case that $q$ is even, and $G = \mathrm{Sp}(2n,\F_q)$, and recall that $G$ is isomorphic to $\SO(2n+1, \F_q)$ in this case (and when $n=0$ we take $\Sp(0, \F_q) = \SO(1, \F_q)$).  Fulman, Guralnick, and Stanton have shown that \cite[Theorems 2.14 and 5.3]{FGS} when $q$ is even, the number of involutions in $\Sp(2n, \F_q)$ is $\prod_{i=1}^{n} (q^{2i} - 1)$ times the coefficient of $u^n$ in the expansion of
$$\frac{1}{1-u}\prod_{i\geq 1}\frac{1 + u/q^{2i}}{ 1-u^2/q^{2i}}.$$
This generating function, together with the Frobenius-Schur involution formula, provides the following criterion for all complex irreducible representations of $\Sp(2n,\F_q)$, with $q$ even, to be defined over the real numbers.

\begin{lemma} \label{Speven}
Let $q$ be even.  Then $\vep(\chi) = 1$ for every $\chi \in \Irr(\Sp(2n,\F_q))$ if and only if the character degree sum of $\Sp(2n, \F_q)$ is $\prod_{i=1}^{n} (q^{2i} - 1)$ times the coefficient of $u^n$ in the expansion of
$$\frac{1}{1-u}\prod_{i \geq 1} \frac{1 + u/q^{2i}}{1-u^2/q^{2i}}.$$
\end{lemma}

Lemmas \ref{SOqodd} and \ref{Speven} motivate our main strategy in proving that indeed $\vep(\chi)=1$ for all $\chi \in \Irr(\Sp(2n,\F_q))$ when $q$ is even.

\subsection{Self-dual polynomials}  \label{selfdualpolys}
Let $f(t) \in \F_q[t]$ be a monic polynomial with nonzero constant term of degree $d$, say $f(t) = t^d + a_{d-1} t^{d-1} + \cdots + a_1 t + a_0$.  We define the {\em dual} polynomial of $f(t)$, denoted $f^*(t)$, as
$$ f^*(t) = a_0^{-1} t^n f(t^{-1}).$$
We say $f(t)$ is {\em self-dual} if $f(t) = f^*(t)$.  Equivalently, $f(t)$ is self-dual if, whenever $\alpha \in \bar{\F}_q^{\times}$ is a root of $f(t)$ with multiplicity $m$, then $\alpha^{-1}$ is also a root of $f(t)$ with multiplicity $m$.

We let $\mathcal{N}(q)$ be the set of monic irreducible self-dual polynomials with nonzero constant in $\F_q[t]$, and we let $N^*(q; d)$ denote the number of polynomials in $\mathcal{N}(q)$ of degree $d$.  Denote by $\mathcal{M}(q)$ the set of unordered pairs $\{g, g^*\}$ of monic irreducible polynomials with nonzero constant in $\F_q[t]$ such that $g(t) \neq g^*(t)$, and let $M^*(q; d)$ be the number of unordered pairs $\{ g, g^*\}$ in $\mathcal{M}(q)$ such that $g$ has degree $d$.  By \cite[Lemma 1.3.16]{FNP}, we have that if $d$ is odd and $d > 1$ then $N^*(q; d)=0$.  That is, other than $t \pm 1$, all polynomials in $\mathcal{N}(q)$ have even degree.

Define, throughout this paper, the number $e(q) = e$ to be $e=1$ if $q$ is even and $e=2$ if $q$ is odd.  We recall the following identities involving self-dual polynomials which we will need.  These results are in \cite[Lemma 1.3.17(a,d)]{FNP}.

\begin{lemma} \label{dualgenfn}
We have the following formal identities of power series in an indeterminate $w$:
\begin{enumerate}
\item[(a)] $\displaystyle \prod_{d \geq 1} (1 - w^d)^{-N^*(q; 2d)} (1 - w^d)^{-M^*(q; d)} = \frac{(1 - w)^e}{1-qw}.$
\item[(b)] $\displaystyle \prod_{d \geq 1} (1 + w^d)^{-N^*(q; 2d)} (1 - w^d)^{-M^*(q; d)} = 1 - w.$
\end{enumerate}
\end{lemma}

\subsection{Partitions and Schur functions} \label{PartSchur}

Let $\mathcal{P}$ denote the set of all partitions of non-negative integers, and let $\mathcal{P}_n$ denote the set of partitions of $n$, where $\mathcal{P}_0$ consists of only the empty partition.  That is, if $\lambda \in \mathcal{P}_n$ with $n \geq 1$, then $\lambda = (\lambda_1, \lambda_2, \ldots, \lambda_l)$ such that $\lambda_i \geq \lambda_{i+1}$ and $\lambda_i > 0$ for each $i$, and 
$$ |\lambda| = \sum_{i = 1}^{l} \lambda_i = n.$$
Given a partition $\lambda = (\lambda_1, \lambda_2, \ldots, \lambda_{\ell}) \in \mathcal{P}$, we will need the statistic $a(\lambda)$ defined as
$$ a(\lambda) = \sum_{i=1}^l (i-1) \lambda_i.$$
Given $\lambda \in \cP$, we let $\lambda'$ denote the conjugate partition, which is obtained by transposing the Young diagram of $\lambda$.  That is, $\lambda_i'$ is the number of $j$ such that $\lambda_j \geq i$.  If we identify $\lambda$ with its Young diagram, let $y \in \lambda$ denote a position in the Young diagram.  We then let $h(y)$ denote the hook-length of that position in $\lambda$, so if $y$ is in position $(i,j)$ of $\lambda$, then $h(y) = \lambda_i + \lambda'_j-i-j+1$ (see \cite[I.1, Example 1]{Mac95}).  From \cite[I.1, Example 2]{Mac95}, we have the identity
\begin{equation} \label{hooks}
\sum_{y \in \lambda} h(y) = a(\lambda) + a(\lambda') + |\lambda|.
\end{equation}

Given $\lambda \in \mathcal{P}$ and a countable set of variables $\{x_1, x_2, \ldots \}$, we let 
$$s_{\lambda} = s_{\lambda}(x_1, x_2, \ldots)$$ 
denote the Schur function corresponding to $\lambda$ in these variables, see \cite[I.3]{Mac95} for a definition.  In particular, $s_{\lambda}$ is a homogeneous symmetric function of degree $|\lambda|$, so for any $b$ we have
\begin{equation} \label{schurhom}
b^{|\lambda|} s_{\lambda}(x_1, x_2, \ldots) = s_{\lambda}(bx_1, bx_2, \ldots).
\end{equation}
We will need to apply the following identities for Schur functions, which are \cite[I.5, Examples 4 and 6]{Mac95}.

\begin{lemma} \label{schurid}
We have the identities
\begin{enumerate}
\item[(1)] $\displaystyle \sum_{\lambda \in \mathcal{P}} s_{\lambda}(x_1, x_2, \ldots) = \prod_{i \geq 1} (1- x_i)^{-1} \prod_{1 \leq i < j} (1 - x_i x_j)^{-1},$ and
\item[(2)] $\displaystyle \sum_{\lambda \in \mathcal{P}} (-1)^{a(\lambda)} s_{\lambda}(x_1, x_2, \ldots) = \prod_{i \geq 1} (1- x_i)^{-1} \prod_{1 \leq i < j} (1 + x_i x_j)^{-1}$.
\end{enumerate}
\end{lemma}

\subsection{Character degrees of finite reductive groups} \label{GenCharDeg}

Let $\bG$ be a connected reductive group over $\bar{\F}_q$, defined over $\F_q$ through a Frobenius automorphism $F$.  In this section, we recall facts about the irreducible complex characters of the finite reductive group $G = \bG^F$.

If $\bT$ is any maximal $F$-stable torus of $\bG$, write $T = \bT^F$, and let $\theta$ be any linear complex character of $T$.  Deligne and Lusztig \cite{DeLu} defined a virtual character $R_T^G(\theta)$ of $G$ corresponding to the pair $(T, \theta)$.  A {\it unipotent} character of $G$ is any irreducible character of $G$ which appears as a constituent of $R_T^G({\bf 1})$ for some maximal torus $T$, where ${\bf 1}$ is the trivial character.  We need the following basic property of unipotent characters.

\begin{lemma} \label{uniprod} Let $\bG, \bG_1,$ and $\bG_2$ be connected reductive groups over $\bar{\F}_q$ such that $\bG \cong \bG_1 \times \bG_2$.  Suppose $\bG$ is defined over $\F_q$ by Frobenius $F$, and $\bG_i$ is $F$-stable for $i=1, 2$, so that $\bG^F \cong \bG_1^{F} \times \bG_2^{F}$.  Then the collection of unipotent characters of $G = \bG^F$ is the collection of characters of the form $\psi_1 \otimes \psi_2$ where $\psi_i$ is a unipotent character of $G_i = \bG_i^{F}$ for $i = 1, 2$.
\end{lemma}
\begin{proof} First, every maximal torus $\bT$ of $\bG$ is of the form $\bT \cong \bT_1 \times \bT_2$, where $\bT_i$ is a maximal torus of $\bG_i$ for $i=1, 2$.  So if $T = \bT^F$ is a maximal $F$-stable torus of $\bG$, we can write $T = T_1 \times T_2$, where $T_i = \bT_i^{F}$ with $\bT_i$ a maximal $F$-stable torus of $\bG_i$ for $i=1, 2$.  Then we have $R_T^G({\bf 1}) = R_{T_1 \times T_2}^{G_1 \times G_2}({\bf 1} \otimes {\bf 1}) \cong R_{T_1}^{G_1} ({\bf 1}) \otimes R_{T_2}^{G_2}({\bf 1})$, by \cite[Lemma 2.6(ii)]{GuLaTi} for example.  The result follows.
\end{proof}

Now let $\bG^*$ be the group dual to $\bG$, with dual Frobenius map $F^*$, and write $G^* = \bG^{*F^*}$ (see \cite[Secs. 4.2 and 4.3]{Ca85}).  We now assume that $\bG$ has connected center.  If $s \in G^*$ is any semisimple element of $G^*$, then the centralizer $C_{\bG^*}(s)$ is a connected reductive group since we are assuming the center $Z(\bG)$ is connected \cite[Theorem 4.5.9]{Ca85}.  So we may consider unipotent characters of the group $C_{\bG^*}(s)^{F^*}$.

With the assumption that $Z(\bG)$ is connected, we have a parametrization of $\Irr(G)$ through the {\it Jordan decomposition} of characters, originally developed by Lusztig \cite{Lu84}.  Given any $\chi \in \Irr(G)$, $\chi$ bijectively corresponds to a $G^*$-conjugacy class of pairs $(s, \psi)$, where $s \in G^*$ is a semisimple element, and $\psi$ is a unipotent character of $C_{\bG^*}(s)^{F^*}$.  We refer to \cite[Chapter 15]{CaEn04} for a thorough treatment of this bijection, which has many useful properties.  The main information we need from this parametrization of characters for $G$ is the description of character degrees, which follows from the main result of Lusztig on the Jordan decomposition of characters \cite[Theorem 4.23]{Lu84} (for relevant discussion see \cite[Section 12.9]{Ca85} and \cite[Remark 13.24]{DiMi}).  Suppose that $\chi \in \Irr(G)$ corresponds to the $G^*$-class of pairs $(s, \psi)$, and write $\chi = \chi_{(s, \psi)}$.  The degree of $\chi_{(s, \psi)}$ is then given by 
\begin{equation} \label{CharDeg}
 \chi_{(s, \psi)}(1) = [G^* : C_{G^*}(s)]_{p'} \psi(1),
\end{equation}
where $p = \mathrm{char}(\F_q)$, and the subscript $p'$ denotes the prime-to-$p$ part.  We remark that, at least when $\bG$ is a simple algebraic group, for any semisimple $s \in G^*$ we have $C_{G^*}(s) = C_{\bG^*}(s)^{F^*}$ (see \cite[pg. 491]{Ca77}).

\section{Character degrees of $\SO(2n+1, \F_q)$} \label{chardegrees}

\subsection{Semisimple classes and centralizers}  \label{SOclasses} We now apply the ideas of Section \ref{GenCharDeg} to the specific case $G = \bG^F = \SO(2n+1, \F_q)$, and note that $\bG = \SO(2n+1, \bar{\F}_q)$ has connected center (since the center is trivial).  In this case, $G^* = \bG^{*F^*} = \Sp(2n, \F_q)$ (see \cite[pg. 120]{Ca85}).  We describe the semisimple classes of $G^* = \Sp(2n,\F_q)$, and unipotent characters of their centralizers.  

To describe the semisimple classes of $\Sp(2n,\F_q)$ and their centralizers, we apply the results of Wall \cite{Wall}.  Specifically, when $q$ is odd this description may be concluded from \cite[Sec. 2.6 Case (B)]{Wall}, and when $q$ is even this is from \cite[Sec. 3.7]{Wall}.  For details and more general results, see \cite[Sec. 2]{Ng10}.

Recall the notation from Section \ref{selfdualpolys} that $\cN(q)$ denotes the set of monic irreducible self-dual polynomials with nonzero constant in $\F_q[t]$, and $\cM(q)$ is the set of unordered pairs $\{g(t), g^*(t)\}$ such that $g \neq g^*$, where $g(t)$ is monic irreducible with nonzero constant in $\F_q[t]$.  Let us further denote $\cN'(q)=\cN(q) \setminus \{ t+1, t-1 \}$, and so if $f(t) \in \cN'(q)$ then $\dgr(f)$ is even.  A semisimple class $(s)$ in $\Sp(2n,\F_q)$ depends only on its elementary divisors, which must each be of the form $f(t)^{m_f}$ for $f(t) \in \cN'(q)$, or $g(t)^{n_g}$ and $(g(t)^*)^{n_g}$ for $\{g(t), g^*(t)\} \in \cM(q)$, or $(t-1)^{2m_+}$ or $(t+1)^{2m_-}$ (and when $q$ is even we only have $(t-1)^{2m_+}$).  In other words, a semisimple class $(s)$ of $\Sp(2n,\F_q)$ is parametrized by a function
\begin{equation} \label{Phidef}
 \Phi: \cN(q) \cup \cM(q) \rightarrow \Z_{\geq 0}, 
\end{equation}
such that, if we write $\Phi(f)=m_f$ for $f \in \cN(q)'$, $\Phi(\{g, g^*\})=m_g$ for $\{g, g^*\} \in \cM(q)$, and $\Phi(t \pm 1)=m_{\pm}$, we have
\begin{equation} \label{classes}
|\Phi|:= \sum_{f \in \cN'(q)} m_f \dgr(f)/2 + \sum_{ \{g, g^*\} \in \cM(q) } m_g \dgr(g) + m_+ + m_- = n,
\end{equation}
where we leave off $m_-$ in the case that $q$ is even.  Given a semisimple class $(s)$ of $G^*=\Sp(2n,\F_q)$ parametrized by $\Phi$ as above, the centralizer $C_{G^*}(s)$ has structure given by
\begin{align}
C_{G^*}(s) \cong \prod_{f \in \cN'(q)} \GU(m_f, &\F_{q^{\dgr(f)/2}}) \times \prod_{ \{ g, g^*\} \in \cM(q)} \GL(m_g, \F_{q^{\dgr(g)}}) \label{cents}\\
& \times \Sp(2m_+, \F_q) \times \Sp(2m_-, \F_q), \nonumber
\end{align}
where the last factor is left off in the case that $q$ is even, and $\GU(n, \F_q)$ denotes the full unitary group defined over $\F_q$.  In reference to the character degree formula given in \eqref{CharDeg}, if $p = \mathrm{char}(\F_q)$, then using the orders of the finite unitary, general linear, and symplectic groups from \eqref{cents} we have
\begin{equation} \label{firstfactor}
[G^*: C_{G^*}(s)]_{p'} = \frac{\prod_{i = 1}^{n} (q^{2i}- 1)}{P(s)},
\end{equation}
where $P(s)$ is given by
$$\prod_{f \in \cN'(q)} \prod_{i=1}^{m_f} (q^{i \dgr(f)/2} - (-1)^i) \prod_{\{g, g^*\} \in \cM(q)} \prod_{i=1}^{m_g} (q^{i \dgr(g)} - 1) \prod_{i=1}^{m_+} (q^{2i} - 1) \prod_{i=1}^{m_-} (q^{2i} - 1),$$
and the last product is left off when $q$ is even.

\subsection{Unipotent characters} \label{Unidegs} We now must describe the unipotent characters of $C_{G^*}(s)$, which by Lemma \ref{uniprod} are products of unipotent characters of the factors of $C_{G^*}(s)$ given in \eqref{cents}.

Unipotent characters of $\GL(n, \F_q)$ and $\GU(n, \F_q)$ are each parametrized by $\cP_n$, the partitions of $n$.  We use the notation of Section \ref{PartSchur} in the following.  Given $\lambda \in \cP_n$, the unipotent character $\psi_{\GL,\lambda}$ of $\GL(n, \F_q)$ parametrized by $\lambda$ has degree which can be written as
$$ \psi_{\GL, \lambda}(1) = \prod_{i=1}^n (q^i - 1) \cdot \frac{ q^{a(\lambda')}}{\prod_{y \in \lambda} (q^{h(y)} - 1)},$$
a form which can be found in \cite[Equation (21)]{Ol86}, and can also be obtained from \cite[Chapter 4, Equation (6.7)]{Mac95}.  From \cite[Remark 9.5]{Lu77}, the unipotent character $\psi_{\GU, \lambda}$ of $\GU(n, \F_q)$ parametrized by $\lambda$ has degree
$$ \psi_{\GU, \lambda}(1) = \prod_{i=1}^n (q^i - (-1)^i) \cdot \frac{q^{a(\lambda')}}{\prod_{y \in \lambda} (q^{h(y)} - (-1)^{h(y)})}.$$
Using \eqref{hooks}, we have
\begin{align*}
\frac{q^{a(\lambda')}}{\prod_{y \in \lambda} (q^{h(y)} - (\pm 1)^{h(y)})} & = \frac{q^{a(\lambda')}}{q^{\sum_{y \in \lambda} h(y)} \prod_{y \in \lambda} (1 - (\pm 1/q)^{h(y)})} \\
& = \frac{1}{q^{|\lambda|} q^{a(\lambda)} \prod_{y \in \lambda} (1 - (\pm 1/q)^{h(y)})} \\
& = \frac{1}{q^n} (\pm 1)^{a(\lambda)} \frac{ (\pm 1/q)^{a(\lambda)}} {\prod_{y \in \lambda} (1 - (\pm 1/q)^{h(y)})}.
\end{align*}
Similar to the calculation in \cite[pg. 286]{Mac95}, we now apply \cite[I.3, Example 2]{Mac95} to write
\begin{equation} \label{GUdeg}
\psi_{\GU, \lambda}(1) = \prod_{i=1}^n (q^i - (-1)^i) \cdot \frac{1}{q^n} (-1)^{a(\lambda)} s_{\lambda} \left(1, -1/q, (-1/q)^2, \ldots \right),
\end{equation} 
and
\begin{equation} \label{GLdeg}
\psi_{\GL, \lambda}(1) = \prod_{i=1}^n (q^i - 1) \cdot \frac{1}{q^n} s_{\lambda}\left(1, 1/q, 1/q^2, \ldots\right).
\end{equation}

In order to describe the unipotent characters of $\Sp(2n, \F_q)$, we need to define {\it symbols}, originally introduced in \cite{Lu77}.  Following \cite[Sec. 13.8]{Ca85}, a symbol is an ordered pair of finite sets of non-negative integers, say $\Lambda = (\mu, \nu)$, where we write $\mu = (\mu_1 < \mu_2 < \ldots < \mu_r)$, $\nu = (\nu_1 < \nu_2 < \ldots < \nu_k)$, such that $\mu_1$ and $\nu_1$ cannot both be $0$, and such that $r-k>0$.  The number $r-k$ is called the {\it defect} of the symbol $\Lambda$.  The {\it rank} of a symbol $\Lambda$, which we denote by $|\Lambda|$, is defined as
$$ |\Lambda| = \sum_{i = 1}^{r} \mu_i + \sum_{i=1}^k \nu_i + \left\lfloor \left( \frac{r+k-1}{2} \right)^2 \right\rfloor.$$
Note that the symbol $(0, \emptyset)$ is the only symbol of rank $0$, and has defect $1$.
We let $\cS$ denote the collection of all symbols of odd defect, and we let $\cS_n$ denote the set of all symbols of odd defect with rank $n$.  It was proved by Lusztig \cite{Lu77} (see also \cite{As83, Lu81}) that the set $\cS_n$ parametrizes the unipotent characters of $\Sp(2n, \F_q)$ as well as those of $\SO(2n+1, \F_q)$ (where we define $\Sp(0, \F_q)=\SO(1, \F_q) = 1$ for any $q$).  If $\psi_{\Lambda}$ is the unipotent character of $\Sp(2n, \F_q)$ (or of $\SO(2n+1, \F_q)$) corresponding to $\Lambda \in \cS_n$, then the degree of $\psi_{\Lambda}$ is given by
\begin{equation} \label{Spdeg}
\psi_{\Lambda}(1) = \prod_{i = 1}^{n} (q^{2i}-1) \cdot \delta(\Lambda),
\end{equation}
where $\delta(\Lambda)$ is given by (following \cite[Sec. 13.8]{Ca85})
$$ \delta(\Lambda) = \frac{ \prod_{i < j} (q^{\mu_i} - q^{\mu_j}) \prod_{i<j} (q^{\nu_i} - q^{\nu_j}) \prod_{i, j} (q^{\mu_i} + q^{\nu_j})}{2^{(r+k-1)/2} q^{c(\Lambda)} \prod_{i = 1}^{r} \prod_{j = 1}^{\lambda_i} (q^{2j} - 1) \prod_{i = 1}^{k} \prod_{j=1}^{\mu_i} (q^{2j} - 1)},$$
with $c(\Lambda) = \sum_{i=1}^{(r+k-1)/2} \binom{r+k-2i}{2}$.  There is also a formula for $\delta(\Lambda)$ in terms of a generalization of hook-lengths for symbols, see \cite{Ma95, Ol86}.  However, we do not need any of these explicit expressions for $\delta(\Lambda)$, and we only mention them here for the sake of context.  The most important fact we need regarding the expression for $\psi_{\Lambda}(1)$ in \eqref{Spdeg} is that it does not depend on whether $q$ is even or odd, and is the same expression in $q$ regardless of the value of $p = \mathrm{char}(\F_q)$ (see \cite[8.11, Remark (1)]{Lu77}, for example).

\subsection{A generating function}  \label{GenFunSec} With $G = \SO(2n+1, \F_q)$, we now consider an arbitrary $\chi \in \Irr(G)$, so $\chi = \chi_{(s, \psi)}$ corresponds to some $G^*$-class of pairs $(s, \psi)$ in the Jordan decomposition.  Then $s \in G^*= \Sp(2n, \F_q)$ corresponds to some $\Phi$ satisfying \eqref{Phidef} and \eqref{classes}, where $C_{G^*}(s)$ has structure given by \eqref{cents}, and there is no $m_-$ term when $q$ is even.  As in Section \ref{Unidegs}, by Lemma \ref{uniprod} the unipotent character $\psi$ then has structure given by
\begin{equation} \label{psistruc}
\psi = \bigotimes_{f \in \cN'(q)}  \psi_{\GU, \lambda_{f}} \otimes \bigotimes_{\{ g, g^*\} \in \cM(q) } \psi_{\GL, \lambda_{g}} \otimes \psi_{\Lambda_+} \otimes \psi_{\Lambda_-},
\end{equation}
where $\psi_{\GU, \lambda_f}$ is a unipotent character of $\GU(m_f, \F_{q^{\dgr(f)/2}})$ with $|\lambda_f| = m_f$, $\psi_{\GL, \lambda_g}$ is a unipotent of $\GL(m_g, \F_{q^{\dgr(g)}})$ with $|\lambda_g|=m_g$, and $\psi_{\Lambda_{\pm}}$ is a unipotent of $\Sp(2m_{\pm}, \F_q)$ with $|\Lambda_{\pm}| = m_{\pm}$ but there is no $m_-$ or $\Lambda_-$ when $q$ is even.

Because of the expressions in \eqref{GUdeg} and \eqref{GLdeg},  we define
$$ \delta_{\GU}(\lambda, d) = \frac{1}{q^{d|\lambda|/2}} (-1)^{a(\lambda)} s_{\lambda} (1, -1/q^{d/2}, (-1/q^{d/2})^2, (-1/q^{d/2})^3, \ldots),$$
$$\text{and } \quad \delta_{\GL}(\lambda, d) = \frac{1}{q^{d |\lambda|}} s_{\lambda}(1, 1/q^d, 1/q^{2d}, 1/q^{3d}, \ldots).$$
Then the degree of $\psi$ in \eqref{psistruc} is given by
$$\psi(1) = P(s) \prod_{f \in \cN'(q)} \delta_{\GU}(\lambda_f, \dgr(f))  \prod_{\{ g, g^*\} \in \cM(q)} \delta_{\GL}(\lambda_g, \dgr(g)) \cdot \delta(\Lambda_+) \delta(\Lambda_-),$$
where $P(s)$ is exactly as in \eqref{firstfactor}, and there are no $m_-$ or $\Lambda_-$ factors when $q$ is even.  Finally, by \eqref{CharDeg} and \eqref{firstfactor}, the degree of $\chi_{(s, \psi)}$ is given by $\prod_{i = 1}^{n} (q^{2i} - 1)$ times
\begin{equation} \label{PreGenFun}
\prod_{f \in \cN'(q)} \delta_{\GU}(\lambda_f, \dgr(f))  \prod_{\{ g, g^*\} \in \cM(q)} \delta_{\GL}(\lambda_g, \dgr(g)) \cdot \delta(\Lambda_+) \delta(\Lambda_-),
\end{equation}
again with no $\delta(\Lambda_-)$ factor when $q$ is even.

We next give a generating function for the sum of the degrees of all irreducible characters of $\SO(2n+1, \F_q)$.  For this purpose, we define the power series $W(u)$ as follows, which is a generating function for the degrees of the unipotent characters of $\Sp(2n, \F_q)$ for any $n \geq 0$:
$$ W(u) = \sum_{\Lambda \in \cS} \delta(\Lambda) u^{|\Lambda|}.$$
The following result gives us the fundamental tool for our main arguments in the next section.

\begin{proposition} \label{GenFunSO}
The character degree sum of $\SO(2n+1, \F_q)$ is $\prod_{i=1}^n (q^{2i} - 1)$ times the coefficient of $u^n$ in the expansion of 
$$\prod_{d \geq 1} \left( \sum_{\lambda \in \cP} \delta_{\GU}(\lambda, 2d) u^{|\lambda| d} \right)^{N^*(q; 2d)} \prod_{d \geq 1} \left (\sum_{\lambda \in \cP} \delta_{\GL}(\lambda, d) u^{|\lambda| d} \right)^{M^*(q; d)} W(u)^e,$$
where $e = 2$ if $q$ is odd and $e=1$ if $q$ is even.
\end{proposition}
\begin{proof}  To obtain such a generating function, we need the coefficient of $u^n$ to be the sum of the expressions in \eqref{PreGenFun}, over all choices for $(s, \psi)$ which parametrize irreducible characters of $\SO(2n+1, \F_q)$.  Since the choice in $(s)$ amounts to the choice of $\Phi$ satisfying \eqref{Phidef} and \eqref{classes}, and the choice of $\psi$ in \eqref{psistruc} amounts to the choices of $\lambda_f, \lambda_g \in \cP$ and $\Lambda_{\pm} \in \cS$, the generating function we want is given by
$$ \prod_{f \in \cN'(q)} \left( \sum_{\lambda_f \in \cP} \delta_{\GU} (\lambda_f, \dgr(f)) u^{|\lambda_f| \dgr(f)/2} \right)$$
$$ \cdot \prod_{ \{ g, g^*\} \in \cM(q)} \left( \sum_{\lambda_g \in \cP} \delta_{\GL} (\lambda_g, \dgr(g)) u^{|\lambda_g| \dgr(g)} \right) $$
$$ \cdot \left( \sum_{\Lambda_+ \in \cS} \delta(\Lambda_+) u^{|\Lambda_+|} \right)  \left( \sum_{\Lambda_- \in \cS} \delta(\Lambda_-) u^{|\Lambda_-|} \right),$$
where the last factor with the sum over $\Lambda_-$ is not included when $q$ is even.  First note that these last two factors are both just $W(u)$, so these factors (or single factor in the $q$ even case) can be replaced by $W(u)^e$.  In the first two products, once $f \in \cN'(q)$ or $\{ g, g^* \} \in \cM(q)$ is fixed, the expression in the summations over $\lambda_f$ or $\lambda_g$ depend only on the degrees of $f$ or $g$.  From these observations, we can rewrite the generating function as
$$\prod_{d \geq 2} \left( \sum_{\lambda \in \cP} \delta_{\GU}(\lambda, d) u^{|\lambda| d/2} \right)^{N^*(q; d)} \prod_{d \geq 1} \left (\sum_{\lambda \in \cP} \delta_{\GL}(\lambda, d) u^{|\lambda| d} \right)^{M^*(q; d)} W(u)^e,$$
where the first product is over $d \geq 2$ since $\cN'(q)$ contains no polynomials of degree 1 by definition.  Further, as stated in Section \ref{selfdualpolys}, we know $\cN'(q)$ contains no polynomials of odd degree.  We can therefore replace $d$ with $2d$ in the first product above and then take the product over all $d \geq 1$, giving the generating function claimed.
\end{proof}

\section{Main Results} \label{Main}

We continue all notation from Section \ref{GenFunSec} above.  We begin with a crucial calculation.

\begin{proposition} \label{GUexpand}
We have the following identity of power series:
\begin{align*}
\prod_{d \geq 1} & \left( \sum_{\lambda \in \cP} \delta_{\GU}(\lambda, 2d) u^{|\lambda| d} \right)^{N^*(q; 2d)} \prod_{d \geq 1} \left (\sum_{\lambda \in \cP} \delta_{\GL}(\lambda, d) u^{|\lambda| d} \right)^{M^*(q; d)} \\
& = \frac{1}{1-u} \prod_{i \geq 1} \frac{(1 - u/q^{2i-1})^e}{1 -u^2/q^{2i}} \prod_{1 \leq i < j \atop{ i + j \text{ odd}}} (1 - u^2/q^{i+j})^e,
\end{align*}
where $e = 2$ if $q$ is odd and $e=1$ if $q$ is even.
\end{proposition}
\begin{proof}  First, we have
\begin{align*}
\delta_{\GU}(\lambda, 2d) u^{|\lambda| d} & = (-1)^{a(\lambda)} \left( \frac{u^d}{q^d} \right)^{|\lambda|} s_{\lambda}\left(1, -\frac{1}{q^d}, \left(-\frac{1}{q^d}\right)^2, \ldots\right) \\
& = (-1)^{a(\lambda)} s_{\lambda}\left(\frac{u^d}{q^d}, \frac{u^d}{q^d} \left( -\frac{1}{q^d} \right), \frac{u^d}{q^d} \left( -\frac{1}{q^d} \right)^2, \ldots \right) \\
& = (-1)^{a(\lambda)} s_{\lambda} (u^d/q^d, -u^d/q^{2d}, u^d/q^{3d}, \ldots ), 
\end{align*}
where we have applied \eqref{schurhom}.  Note that the variables of the Schur function $s_{\lambda}$ above are $(-1)^{i-1}u^d/q^{id}$ for $i \geq 1$.  By Lemma \ref{schurid}(2), we then have
\begin{align}
 \sum_{\lambda \in \cP} \delta_{\GU}(\lambda, 2d) u^{|\lambda| d}  &= \sum_{\lambda \in \cP} (-1)^{a(\lambda)} s_{\lambda} (u^d/q^d, -u^d/q^{2d}, u^d/q^{3d}, \ldots)  \nonumber\\
& = \prod_{i \geq 1} \left(1 - (-1)^{i-1} \frac{u^d}{q^{id}}\right)^{-1} \prod_{1 \leq i < j} \left( 1 + (-1)^{i+j} \frac{u^{2d}}{q^{(i+j)d}} \right)^{-1}. \label{Uexp}
\end{align}
Next we have, again by \eqref{schurhom},
\begin{align*}
\delta_{\GL}(\lambda, d)u^{|\lambda|d} & = \left(\frac{u^d}{q^d}\right)^{|\lambda|} s_{\lambda}(1, 1/q^d, 1/q^{2d}, \ldots) \\
& = s_{\lambda}(u^d/q^d, u^d/q^{2d}, u^d/q^{3d}, \ldots).
\end{align*}
Noting that variables of the Schur function are $u^d/q^{id}$, we may apply Lemma \ref{schurid}(1) to obtain
\begin{align}
\sum_{\lambda \in \cP} \delta_{\GL}(\lambda, d) u^{|\lambda| d} & = \sum_{\lambda \in \cP} s_{\lambda} (u^d/q^d, u^d/q^{2d}, u^d/q^{3d}, \ldots) \nonumber \\
& = \prod_{i \geq 1} \left(1 - \frac{u^d}{q^{id}} \right)^{-1} \prod_{1 \leq i < j} \left( 1 - \frac{u^{2d}}{q^{(i+j)d}} \right)^{-1}. \label{Gexp}
\end{align}
Now apply \eqref{Uexp} and \eqref{Gexp} and rearrange factors:
\begin{align}
\prod_{d \geq 1} & \left( \sum_{\lambda \in \cP} \delta_{\GU}(\lambda, 2d) u^{|\lambda| d} \right)^{N^*(q; 2d)} \prod_{d \geq 1} \left (\sum_{\lambda \in \cP} \delta_{\GL}(\lambda, d) u^{|\lambda| d} \right)^{M^*(q; d)} \nonumber\\
&= \prod_{d \geq 1} \left[ \prod_{i \geq 1} \left(1 - (-1)^{i-1} \frac{u^d}{q^{id}}\right) \prod_{1 \leq i < j} \left(1 + (-1)^{i+j} \frac{u^{2d}}{q^{(i+j)d}} \right)\right]^{-N^*(q; 2d)} \nonumber \\
& \quad \cdot \prod_{d \geq 1} \left[ \prod_{i \geq 1} \left(1 - \frac{u^d}{q^{id}}\right) \prod_{1 \leq i < j} \left(1  -\frac{u^{2d}}{q^{(i+j)d}} \right)\right]^{-M^*(q; d)} \nonumber \\
& = \prod_{i \geq 1 \atop{ i \text{ odd}}} \left[ \prod_{d \geq 1} ( 1 - u^d/q^{id})^{-N^*(q; 2d)} (1 - u^d/q^{id})^{-M^*(q; d)} \right] \label{factor1} \\
& \quad \cdot \prod_{i \geq 1 \atop{ i \text{ even}}} \left[ \prod_{d \geq 1} ( 1 + u^d/q^{id})^{-N^*(q; 2d)} (1 - u^d/q^{id})^{-M^*(q; d)} \right] \label{factor2} \\
& \quad \cdot \prod_{1 \leq i < j \atop{ i +j \text{ odd}}} \left[ \prod_{d \geq 1} ( 1 - u^{2d}/q^{(i+j)d})^{-N^*(q; 2d)} (1 - u^{2d}/q^{(i+j)d})^{-M^*(q; d)} \right] \label{factor3} \\
& \quad \cdot \prod_{1 \leq i < j \atop{ i +j \text{ even}}} \left[ \prod_{d \geq 1} ( 1 + u^{2d}/q^{(i+j)d})^{-N^*(q; 2d)} (1 - u^{2d}/q^{(i+j)d})^{-M^*(q; d)} \right]. \label{factor4} 
\end{align}
Apply Lemma \ref{dualgenfn}(1) to \eqref{factor1} with $w=u/q^i$, Lemma \ref{dualgenfn}(2) to \eqref{factor2} with $w =u/q^i$, Lemma \ref{dualgenfn}(1) to \eqref{factor3} with $w=u^2/q^{i+j}$, and Lemma \ref{dualgenfn}(2) to \eqref{factor4} with $w=u^2/q^{i+j}$ to get
\begin{align}
\prod_{d \geq 1} & \left( \sum_{\lambda \in \cP} \delta_{\GU}(\lambda, 2d) u^{|\lambda| d} \right)^{N^*(q; 2d)} \prod_{d \geq 1} \left (\sum_{\lambda \in \cP} \delta_{\GL}(\lambda, d) u^{|\lambda| d} \right)^{M^*(q; d)} \label{almost}\\
&= \prod_{i \geq 1 \atop{ i \text{ odd}}} \frac{(1-u/q^i)^e}{1 - u/q^{i-1}} \prod_{i \geq 1 \atop{ i \text{ even}}} (1 - u/q^i) \nonumber \\
& \quad \quad \cdot \prod_{1 \leq i < j \atop{ i + j \text{ odd}}} \frac{(1 - u^2/q^{i+j})^e}{1 - u^2/q^{i+j-1}} \prod_{1 \leq i < j \atop{ i+j \text{ even}}} (1 - u^2/q^{i+j}). \nonumber
\end{align}
Now note that
$$ \prod_{i \geq 1 \atop{ i \text{ odd}}} \frac{1}{1 -u/q^{i-1}} \prod_{i \geq 1 \atop{ i \text{ even}}} (1 - u/q^i) = \frac{1}{1-u},$$
so 
\begin{equation} \label{almost1} 
\prod_{i \geq 1 \atop{ i \text{ odd}}} \frac{(1-u/q^i)^e}{1 - u/q^{i-1}} \prod_{i \geq 1 \atop{ i \text{ even}}} (1 - u/q^i) = \frac{1}{1-u} \prod_{i \geq 1} (1 - u/q^{2i-1})^e. 
\end{equation}
We also have
$$ \prod_{1 \leq i < j \atop{ i+j \text{ odd}}} \frac{1}{1 - u^2/q^{i+j-1}} \prod_{1 \leq i < j \atop{ i +j \text{ even}}} (1 - u^2/q^{i+j}) = \prod_{i \geq 1} \frac{1}{1 - u^2/q^{2i}},$$
since the only factors which do not cancel are from the first product when $j=i+1$.  Thus
\begin{align} 
\prod_{1 \leq i < j \atop{ i + j \text{ odd}}} & \frac{(1 - u^2/q^{i+j})^e}{1 - u^2/q^{i+j-1}} \prod_{1 \leq i < j \atop{ i+j \text{ even}}} (1 - u^2/q^{i+j}) \label{almost2}\\
& = \prod_{i \geq 1} \frac{1}{1 - u^2/q^{2i}} \prod_{1 \leq i < j \atop{ i + j \text{ odd}}} (1 - u^2/q^{i+j})^e. \nonumber
\end{align}
Finally, substituting \eqref{almost1} and \eqref{almost2} into \eqref{almost} gives the desired identity.
\end{proof}

We are now in the position to prove the main results.

\begin{theorem} \label{main1}
For any prime power $q$, the sum of the degrees of the unipotent characters of $\Sp(2n, \F_q)$, or of $\SO(2n+1, \F_q)$, is $\prod_{i=1}^n (q^{2i} - 1)$ times the coefficient of $u^n$ in the expansion of
$$ \prod_{i \geq 1} \frac{ 1 + u/q^{2i} }{1 - u/q^{2i-1}} \prod_{1 \leq i < j \atop{ i + j \text{ odd}}} \frac{1}{1 - u^2/q^{i+j}}.$$
\end{theorem}
\begin{proof}  The generating function which we are calculating is
$$ W(u) = \sum_{\Lambda \in \cS} \delta(\Lambda) u^{|\Lambda|} = \sum_{n \geq 0} \left( \sum_{\Lambda \in \cS_n} \delta(\Lambda) \right) u^n,$$
where $\delta(\Lambda)$ is the expression in $q$ given in Section \ref{Unidegs}.  By Lemma \ref{SOqodd} and Proposition \ref{GenFunSO}, we have that when $q$ is odd
\begin{align*}
\prod_{d \geq 1} & \left( \sum_{\lambda \in \cP} \delta_{\GU}(\lambda, 2d) u^{|\lambda| d} \right)^{N^*(q; 2d)} \prod_{d \geq 1} \left (\sum_{\lambda \in \cP} \delta_{\GL}(\lambda, d) u^{|\lambda| d} \right)^{M^*(q; d)} W(u)^2 \\
& = \frac{1}{1-u} \frac{\prod_{i \geq 1} (1 + u/q^{2i})^2}{\prod_{i \geq 1} (1 - u^2/q^{2i})}.
\end{align*}
By substituting in the expression from Proposition \ref{GUexpand} with $q$ odd (so $e=2$), we have
$$\frac{W(u)^2}{1-u} \prod_{i \geq 1} \frac{(1 - u/q^{2i-1})^2}{1 -u^2/q^{2i}} \prod_{1 \leq i < j \atop{ i + j \text{ odd}}} (1 - u^2/q^{i+j})^2 =  \frac{1}{1-u} \frac{\prod_{i \geq 1} (1 + u/q^{2i})^2}{\prod_{i \geq 1} (1 - u^2/q^{2i})}.$$
Solving for $W(u)$ yields
$$W(u) = \prod_{i \geq 1} \frac{1 + u/q^{2i}}{1-u/q^{2i-1}} \prod_{1 \leq i < j \atop{ i + j \text{ odd}}} \frac{1}{1 - u^2/q^{i+j}}.$$
Since each $\delta(\Lambda)$ is an expression in $q$ which does not depend on the parity of $q$, then this expression for $W(u)$ holds for any prime power $q$.  Since the unipotent characters for $\Sp(2n,\F_q)$ and $\SO(2n+1,\F_q)$ are parametrized by the same symbols, and have the same corresponding degrees, then $W(u)$ is a generating function for the sum of the degrees of the unipotent characters of $\SO(2n+1,\F_q)$ as well.
\end{proof}

Note that it follows from \eqref{Uexp} that the sum of the degrees of the unipotent characters of $\GU(n, \F_q)$ is $\prod_{i = 1}^n (q^i - (-1)^i)$ times the coefficient of $u^n$ in 
$$\prod_{i \geq 1} \frac{1}{1 - (-1)^{i-1} u/q^{i}} \prod_{1 \leq i < j} \frac{1}{1 + (-1)^{i+j} u^{2}/q^{i+j}},$$
and from \eqref{Gexp} the sum of the degrees of the unipotent characters of $\GL(n, \F_q)$ is $\prod_{i = 1}^n (q^i - 1)$ times the coefficient of $u^n$ in 
$$\prod_{i \geq 1} \frac{1}{1 - u/q^{i}} \prod_{1 \leq i < j} \frac{1}{1 -  u^{2}/q^{i+j}},$$
and these are each direct consequences of the Schur function identities in Lemma \ref{schurid}.  It would of course be satisfying to have such a direct combinatorial proof of Theorem \ref{main1}.

\begin{theorem} \label{main2}
Let $q$ be a power of $2$, and $G = \Sp(2n,\F_q)$.  Then $\varepsilon(\chi) = 1$ for every $\chi \in \Irr(G)$, that is, every complex irreducible representation of $G$ may be defined over the real numbers.
\end{theorem}
\begin{proof} When $q$ is a power of $2$, we have $G =\Sp(2n,\F_q) \cong \SO(2n+1, \F_q)$, and so by Proposition \ref{GenFunSO} the character degree sum of $G$ (with $e=1$ since $q$ is even) is $\prod_{i = 1}^n (q^{2i} - 1)$ times the coefficient of $u^n$ in the power series
$$\prod_{d \geq 1} \left( \sum_{\lambda \in \cP} \delta_{\GU}(\lambda, 2d) u^{|\lambda| d} \right)^{N^*(q; 2d)} \prod_{d \geq 1} \left (\sum_{\lambda \in \cP} \delta_{\GL}(\lambda, d) u^{|\lambda| d} \right)^{M^*(q; d)} W(u).$$
Substitute the expression from Proposition \ref{GUexpand} for the first two products (with $e=1$) and the expression for $W(u)$ from Theorem \ref{main1} to find that the character degree sum of $G$ is $\prod_{i=1}^n (q^{2i}-1)$ times the coefficient of $u^n$ in the expansion of 
$$\frac{1}{1-u} \prod_{i \geq 1} \frac{1 - u/q^{2i-1}}{1 -u^2/q^{2i}} \prod_{1 \leq i < j \atop{ i + j \text{ odd}}} (1 - u^2/q^{i+j}) \prod_{i \geq 1} \frac{ 1 + u/q^{2i} }{1 - u/q^{2i-1}} \prod_{1 \leq i < j \atop{ i + j \text{ odd}}} \frac{1}{1 - u^2/q^{i+j}}$$
$$ = \frac{1}{1 - u} \prod_{i \geq 1} \frac{ 1 + u/q^{2i}}{ 1 - u^2/q^{2i}}.$$
The result now follows from Lemma \ref{Speven}
\end{proof}

\end{document}